\documentclass{amsart}
\usepackage{amssymb}
\usepackage{amsmath}
\usepackage{amsfonts}

\newtheorem{theorem}{Theorem}
\theoremstyle{plain}

\newtheorem{corollary}{Corollary}

\newtheorem{definition}{Definition}

\newtheorem{remark}{Remark}

\numberwithin{equation}{section}
\input{tcilatex}

\begin{document}
\title[Two Capacities and Thinness in Weighted Variable Exponent Sobolev
Spaces]{On Some Properties of Relative Capacity and Thinness in Weighted
Variable Exponent Sobolev Spaces}
\author{CIHAN UNAL}
\address{Sinop University\\
Faculty of Arts and Sciences\\
Department of Mathematics}
\email{cihanunal88@gmail.com}
\urladdr{}
\thanks{}
\author{ISMAIL AYDIN}
\address{Sinop University\\
Faculty of Arts and Sciences\\
Department of Mathematics}
\email{iaydin@sinop.edu.tr}
\urladdr{}
\thanks{}
\dedicatory{}
\thanks{}

\begin{abstract}
In this paper, we define weighted relative $p(.)$-capacity and discuss
properties of capacity in the space $W_{\vartheta }^{1,p(.)}(%
\mathbb{R}
^{n}).$ Also, we investigate some properties of weighted variable Sobolev
capacity. It is shown that there is a relation between these two capacities.
Moreover, we introduce a thinness in sense to this new defined relative
capacity and prove an equivalence statement for this thinness.
\end{abstract}

\maketitle

\section{Introduction}

Kov\'{a}\v{c}ik and R\'{a}kosn\'{\i}k \cite{K} introduced the variable
exponent Lebesgue space $L^{p\left( .\right) }(%
\mathbb{R}
^{n})$ and the Sobolev space $W^{k,p(.)}\left( 
\mathbb{R}
^{n}\right) $. The boundedness of the maximal operator was an open problem
in $L^{p\left( .\right) }(%
\mathbb{R}
^{n})$ for a long time. Diening \cite{Dien} proved the first time this state
over bounded domains if $p\left( .\right) $ satisfies locally log-H\"{o}lder
continuous condition, that is,%
\begin{equation*}
\left\vert p\left( x\right) -p\left( y\right) \right\vert \leq \frac{C}{-\ln
\left\vert x-y\right\vert }\text{, }x,y\in \Omega ,\text{ }\left\vert
x-y\right\vert \leq \frac{1}{2}
\end{equation*}%
where $\Omega $ is a bounded domain. We denote by $P^{\log }\left( 
\mathbb{R}
^{n}\right) $ the class of variable exponents which satisfy the log-H\"{o}%
lder continuous condition. Diening later extended the result to unbounded
domains by supposing, in addition, that the exponent $p\left( .\right) =p$
is a constant function outside a large ball. After this study, many
absorbing and crucial papers appeared in non-weighted and weighted variable
exponent spaces. For a historical journey, we refer \cite{Dien2}, \cite{Fan}%
, \cite{K}, \cite{Mus} and \cite{Sam}. Sobolev capacity for constant
exponent spaces has found a great number of uses, see \cite{Eva} and \cite%
{Maz}. Moreover, the weighted Sobolev capacity was revealed by Kilpel\"{a}%
inen \cite{Kil}. He investigated the role of capacity in the pointwise
definition of functions in Sobolev spaces involving weights of Muckenhoupt's 
$A_{p}-$class. Harjulehto et al. \cite{Har2} introduced variable Sobolev
capacity in the spaces $W^{1,p(.)}\left( 
\mathbb{R}
^{n}\right) .$ Also, Ayd\i n \cite{A2} generalized some results of the
variable Sobolev capacity to the weighted variable exponent case.

The variational capacity has been used extensively in nonlinear potential
theory on $%
\mathbb{R}
^{n}$. Let \ $\Omega \subset 
\mathbb{R}
^{n}$ is open and $K\subset \Omega $ is compact. Then the relative
variational $p$-capacity is defined by%
\begin{equation*}
\text{cap}_{p}\left( K,\Omega \right) =\inf_{f}\dint\limits_{\Omega
}\left\vert \bigtriangledown f\left( x\right) \right\vert ^{p}dx\text{,}
\end{equation*}%
where the infimum is taken over smooth and zero boundary valued functions $f$
in $\Omega $ such that $f\geq 1$ in $K.$ The set of admissible functions $f$
can be replaced by the continuous first order Sobolev functions with $f\geq
1 $ in $K.$ The $p$-capacity is a Choquet capacity relative to $\Omega .$
For more details and historical background, see \cite{Hei}. Also, Harjulehto
et al. \cite{Har1} defined a relative capacity. They studied properties of
the capacity and compare it with the Sobolev capacity.

In \cite{Dubinin}, the authors have considered a relative capacity and
relationship with the well-known half-plane capacity. It is known that the
half-plane capacity is a particular case because of its applications in
geometric function theory and stochastic processes. Also, they proved some
properties of defined relative capacity such as the behavior of this
capacity under various forms of symmetrization and under some other
geometric transformations. Moreover, they investigated some applications to
bounded holomorphic functions of the unit disk.

Our purpose is to investigate some properties of the Sobolev capacity and,
also, relative $p(.)$-capacity in sense to Harjulehto et al. \cite{Har1} to
the weighted variable exponent case. Also, we give relationship between
these defined two capacities. Moreover, we present a thinness in sense to
this new defined relative capacity and prove an equivalence statement for
this thinness.

\section{Notation and Preliminaries}

In this paper, we will work on $%
\mathbb{R}
^{n}$ with Lebesgue measure $dx$. The measure $\mu $ is doubling if there is
a fixed constant $c_{d}\geq 1,$ called the doubling constant of $\mu $ such
that%
\begin{equation*}
\mu \left( B\left( x_{0},2r\right) \right) \leq c_{d}\mu \left( B\left(
x_{0},r\right) \right)
\end{equation*}%
for every ball $B\left( x_{0},r\right) $ in $%
\mathbb{R}
^{n}.$ Also, the elements of the space $C_{0}^{\infty }(%
\mathbb{R}
^{n})$ are the infinitely differentiable functions with compact support. We
denote the family of all measurable functions $p(.):%
\mathbb{R}
^{n}\rightarrow \lbrack 1,\infty )$ (called the variable exponent on $%
\mathbb{R}
^{n}$) by the symbol $\mathcal{P}\left( 
\mathbb{R}
^{n}\right) $. In this paper, the function $p(.)$ always denotes a variable
exponent. For $p(.)\in \mathcal{P}\left( 
\mathbb{R}
^{n}\right) ,$ put 
\begin{equation*}
p^{-}=\underset{x\in 
\mathbb{R}
^{n}}{\text{ess inf}}p(x)\text{, \ \ \ \ \ \ }p^{+}=\underset{x\in 
\mathbb{R}
^{n}}{\text{ess sup}}p(x)\text{.}
\end{equation*}

A positive, measurable and locally integrable function $\vartheta :%
\mathbb{R}
^{n}\rightarrow \left( 0,\infty \right) $ is called a weight function. The
weighted modular is defined by 
\begin{equation*}
\rho _{p(.),\vartheta }(f)=\dint\limits_{%
\mathbb{R}
^{n}}\left\vert f(x)\right\vert ^{p(x)}\vartheta \left( x\right) dx\text{.}
\end{equation*}%
The weighted variable exponent Lebesgue spaces $L_{\vartheta }^{p(.)}(%
\mathbb{R}
^{n})$ consist of all measurable functions $f$ on $%
\mathbb{R}
^{n}$ endowed with the Luxemburg norm%
\begin{equation*}
\left\Vert f\right\Vert _{p\left( .\right) ,\vartheta }=\inf \left\{ \lambda
>0:\dint\limits_{%
\mathbb{R}
^{n}}\left\vert \frac{f(x)}{\lambda }\right\vert ^{p(x)}\vartheta \left(
x\right) dx\leq 1\right\} .
\end{equation*}%
When $\vartheta \left( x\right) =1,$ the space $L_{\vartheta }^{p(.)}(%
\mathbb{R}
^{n})$ is the variable exponent Lebesgue space. The space $L_{\vartheta
}^{p(.)}(%
\mathbb{R}
^{n})$ is a Banach space with respect to $\left\Vert .\right\Vert
_{p(.),\vartheta }.$ Also, some basic properties of this space were
investigated in \cite{A1}, \cite{A2}, \cite{Kok}.

Let $\Omega \subset 
\mathbb{R}
^{n}$ is bounded and $\vartheta $ is a weight function$.$ It is known that a
function $f\in C_{0}^{\infty }\left( \Omega \right) $ satisfy Poincar\'{e}
inequality in $L_{\vartheta }^{1}(\Omega )$ if and only if there is a
constant $c>0$ such that the inequality%
\begin{equation*}
\dint\limits_{\Omega }\left\vert f(x)\right\vert \vartheta \left( x\right)
dx\leq c\left( \text{diam }\Omega \right) \dint\limits_{\Omega }\left\vert
\nabla f(x)\right\vert \vartheta \left( x\right) dx
\end{equation*}%
holds \cite{Hei}.

In recent decades, variable exponent Lebesgue spaces $L^{p(.)}$ and the
corresponding the variable exponent Sobolev spaces $W^{k,p(.)}$ have
attracted more and more attention. Let $1<p^{-}\leq p\left( .\right) \leq
p^{+}<\infty $ and $k\in 
\mathbb{N}
.$ The variable exponent Sobolev spaces $W^{k,p(.)}\left( 
\mathbb{R}
^{n}\right) $ consist of all measurable functions $f\in L^{p(.)}(%
\mathbb{R}
^{n})$ such that the distributional derivatives $D^{\alpha }f$ are in $%
L^{p(.)}(%
\mathbb{R}
^{n})$ for all $0\leq \left\vert \alpha \right\vert \leq k$ where $\alpha
\in 
\mathbb{N}
_{0}^{n}$ is a multiindex, $\left\vert \alpha \right\vert =\alpha
_{1}+\alpha _{2}+...+\alpha _{n},$ and $D^{\alpha }=\frac{\partial
^{\left\vert \alpha \right\vert }}{\partial _{x_{1}}^{\alpha _{1}}\partial
_{x_{2}}^{\alpha _{2}}...\partial _{x_{n}}^{\alpha _{n}}}.$ The spaces $%
W^{k,p(.)}\left( 
\mathbb{R}
^{n}\right) $ are a special class of so-called generalized Orlicz-Sobolev
spaces with the norm%
\begin{equation*}
\left\Vert f\right\Vert _{k,p\left( .\right) }=\sum_{0\leq \left\vert \alpha
\right\vert \leq k}\left\Vert D^{\alpha }f\right\Vert _{p\left( .\right) }.
\end{equation*}

We set the weighted variable exponent Sobolev spaces $W_{\vartheta
}^{k,p\left( .\right) }\left( 
\mathbb{R}
^{n}\right) $ by

\begin{equation*}
W_{\vartheta }^{k,p(.)}(%
\mathbb{R}
^{n})=\left\{ f\in L_{\vartheta }^{p\left( .\right) }\left( 
\mathbb{R}
^{n}\right) :D^{\alpha }f\in L_{\vartheta }^{p(.)}(%
\mathbb{R}
^{n}),0\leq \left\vert \alpha \right\vert \leq k\right\}
\end{equation*}%
equipped with the norm

\begin{equation*}
\left\Vert f\right\Vert _{k,p\left( .\right) ,\vartheta }=\sum_{0\leq
\left\vert \alpha \right\vert \leq k}\left\Vert D^{\alpha }f\right\Vert
_{p\left( .\right) ,\vartheta }.
\end{equation*}%
It is already known that $W_{\vartheta }^{k,p(.)}(%
\mathbb{R}
^{n})$ is a reflexive Banach space.

Now, let $1<p^{-}\leq p\left( .\right) \leq p^{+}<\infty $, $k\in 
\mathbb{N}
$ and $\vartheta ^{-\frac{1}{p\left( .\right) -1}}\in L_{loc}^{1}\left( 
\mathbb{R}
^{n}\right) .$ Thus, we have $L_{\vartheta }^{p\left( .\right) }\left( 
\mathbb{R}
^{n}\right) \hookrightarrow L_{loc}^{1}\left( 
\mathbb{R}
^{n}\right) $ and then the weighted variable exponent Sobolev spaces $%
W_{\vartheta }^{k,p\left( .\right) }\left( 
\mathbb{R}
^{n}\right) $ is well-defined by [\cite{A2}, Proposition 2.1].

In particular, the space $W_{\vartheta }^{1,p\left( .\right) }\left( 
\mathbb{R}
^{n}\right) $ is defined by%
\begin{equation*}
W_{\vartheta }^{1,p(.)}\left( 
\mathbb{R}
^{n}\right) =\left\{ f\in L_{\vartheta }^{p\left( .\right) }\left( 
\mathbb{R}
^{n}\right) :\left\vert \nabla f\right\vert \in L_{\vartheta }^{p(.)}(%
\mathbb{R}
^{n})\right\} .
\end{equation*}%
The function $\rho _{1,p\left( .\right) ,\vartheta }:W_{\vartheta }^{1,p(.)}(%
\mathbb{R}
^{n})\longrightarrow \left[ 0,\infty \right) $ is shown as $\rho _{1,p\left(
.\right) ,\vartheta }\left( f\right) =\rho _{p\left( .\right) ,\vartheta
}\left( f\right) +\rho _{p\left( .\right) ,\vartheta }\left( \left\vert
\nabla f\right\vert \right) .$ Also, the norm $\left\Vert f\right\Vert
_{1,p\left( .\right) ,\vartheta }=\left\Vert f\right\Vert _{p\left( .\right)
,\vartheta }+\left\Vert \nabla f\right\Vert _{p\left( .\right) ,\vartheta }$
makes the space $W_{\vartheta }^{1,p\left( .\right) }\left( 
\mathbb{R}
^{n}\right) $ a Banach space. The local weighted variable exponent Sobolev
space $W_{\vartheta ,loc}^{1,p\left( .\right) }\left( 
\mathbb{R}
^{n}\right) $ is defined in the classical way. More information on the
classic theory of variable exponent spaces can be found in \cite{K}.

As an alternative to the Sobolev $p(.)$- capacity, Harjulehto et al. \cite%
{Har1} introduced relative $p(.)$- capacity. Recall that%
\begin{equation*}
C_{0}(\Omega )=\left\{ f:\Omega \longrightarrow 
\mathbb{R}
:f\text{ is continuous and supp}f\subset \Omega \text{ is compact}\right\} ,
\end{equation*}%
where supp$f$ is the support of $f$. Suppose that $K$ is a compact subset of 
$\Omega .$ We denote%
\begin{equation*}
R_{p\left( .\right) }\left( K,\Omega \right) =\left\{ f\in W^{1,p(.)}\left(
\Omega \right) \cap C_{0}\left( \Omega \right) :f\geq 1\text{ on }K\right\}
\end{equation*}%
and define%
\begin{equation*}
\text{cap}_{p\left( .\right) }^{\ast }\left( K,\Omega \right) =\inf_{f\in
R_{p\left( .\right) }\left( K,\Omega \right) }\dint\limits_{\Omega
}\left\vert \bigtriangledown f\left( x\right) \right\vert ^{p\left( x\right)
}dx=\inf_{f\in R_{p\left( .\right) }\left( K,\Omega \right) }\rho _{p\left(
.\right) }\left( \left\vert \bigtriangledown f\right\vert \right) .
\end{equation*}%
Further, if $U\subset \Omega $ is open, then%
\begin{equation*}
\text{cap}_{p\left( .\right) }\left( U,\Omega \right) =\sup_{\substack{ %
K\subset U  \\ compact}}\text{cap}_{p\left( .\right) }^{\ast }\left(
K,\Omega \right) ,
\end{equation*}%
and for an arbitrary set $E\subset \Omega $%
\begin{equation*}
\text{cap}_{p\left( .\right) }\left( E,\Omega \right) =\inf_{_{\substack{ %
E\subset U\subset \Omega  \\ U\text{ open}}}}\text{cap}_{p\left( .\right)
}\left( U,\Omega \right) .
\end{equation*}%
The number cap$_{p\left( .\right) }\left( E,\Omega \right) $ is called the
variational $p(.)$- capacity of $E$ relative to $\Omega .$ It is usually
called simply the relative $p(.)$- capacity of the pair or condenser $\left(
E,\Omega \right) .$

Throughout this paper, we assume that $p\left( .\right) \in P^{\log }\left( 
\mathbb{R}
^{n}\right) $ with $1<p^{-}\leq p\left( .\right) \leq p^{+}<\infty $ and $%
\vartheta ^{-\frac{1}{p\left( .\right) -1}}\in L_{loc}^{1}\left( 
\mathbb{R}
^{n}\right) .$ We write that $a\approx b$ for two quantities if there exists
positive constants $c_{1},c_{2}$ such that $c_{1}a\leq b\leq c_{2}a.$ Also,
we will denote%
\begin{equation*}
\mu _{\vartheta }\left( \Omega \right) =\dint\limits_{\Omega }\vartheta
\left( x\right) dx.
\end{equation*}

\section{The Sobolev $\left( p\left( .\right) ,\protect\vartheta \right) $%
-Capacity and The Relative $\left( p\left( .\right) ,\protect\vartheta %
\right) $- Capacity}

A capacity for subsets of $%
\mathbb{R}
^{n}$ was introduced in \cite{A2}. To define this capacity we denote%
\begin{equation*}
S_{p\left( .\right) ,\vartheta }(E)=\left\{ f\in W_{\vartheta
}^{1,p(.)}\left( 
\mathbb{R}
^{n}\right) :f\geq 1\text{ in open set containing }E\right\} .
\end{equation*}%
The Sobolev $\left( p\left( .\right) ,\vartheta \right) $- capacity of $E$
is defined by%
\begin{equation*}
C_{p\left( .\right) ,\vartheta }\left( E\right) =\inf_{f\in S_{p\left(
.\right) ,\vartheta }(E)}\rho _{1,p\left( .\right) ,\vartheta }\left(
f\right) .
\end{equation*}%
Thanks to meaning of the infimum, in case $S_{p\left( .\right) ,\vartheta
}(E)=\emptyset ,$ we set $C_{p\left( .\right) ,\vartheta }\left( E\right)
=\infty $. If $1<p^{-}\leq p\left( .\right) \leq p^{+}<\infty ,$ then the
set function $E\longrightarrow C_{p\left( .\right) ,\vartheta }\left(
E\right) $ is an outer measure. If $f\in S_{p\left( .\right) ,\vartheta
}(E), $ then $\min \left\{ 1,f\right\} \in S_{p\left( .\right) ,\vartheta
}(E)$ and $\rho _{1,p\left( .\right) ,\vartheta }\left( \min \left\{
1,f\right\} \right) \leq \rho _{1,p\left( .\right) ,\vartheta }\left(
f\right) .$ Thus it is enough to test the Sobolev $\left( p\left( .\right)
,\vartheta \right) $- capacity by $f\in S_{p\left( .\right) ,\vartheta }(E)$
with $0\leq f\leq 1.$

\begin{remark}
In general, it is known that the space $C^{\infty }\left( 
\mathbb{R}
^{n}\right) \cap W_{\vartheta }^{1,p\left( .\right) }\left( 
\mathbb{R}
^{n}\right) $ is not dense in $W_{\vartheta }^{1,p\left( .\right) }\left( 
\mathbb{R}
^{n}\right) $. But Zhikov and Surnachev have investigated a sufficient
condition for this denseness. This condition was formulated in terms of the
asymptotic behavior of the integrals of negative and positive powers of the
weight, see \cite{Zhi}. In this paper, we will assume that this denseness
holds.
\end{remark}

\begin{theorem}
\label{Yogunluk}Assume that $1<p^{-}\leq p\left( .\right) \leq p^{+}<\infty $
and $C^{\infty }\left( 
\mathbb{R}
^{n}\right) \cap W_{\vartheta }^{1,p\left( .\right) }\left( 
\mathbb{R}
^{n}\right) $ is dense in $W_{\vartheta }^{1,p\left( .\right) }\left( 
\mathbb{R}
^{n}\right) $. If K is compact, then%
\begin{equation*}
C_{p\left( .\right) ,\vartheta }\left( K\right) =\inf_{f\in S_{p\left(
.\right) ,\vartheta }^{\infty }(K)}\rho _{1,p\left( .\right) ,\vartheta
}\left( f\right)
\end{equation*}%
where $S_{p\left( .\right) ,\vartheta }^{\infty }(K)=S_{p\left( .\right)
,\vartheta }(K)\cap C^{\infty }\left( 
\mathbb{R}
^{n}\right) .$
\end{theorem}

\begin{proof}
Given any $f\in S_{p\left( .\right) ,\vartheta }(K)$ with $0\leq f\leq 1.$
Since by the assumption $C^{\infty }\left( 
\mathbb{R}
^{n}\right) \cap W_{\vartheta }^{1,p\left( .\right) }\left( 
\mathbb{R}
^{n}\right) $ is dense in $W_{\vartheta }^{1,p\left( .\right) }\left( 
\mathbb{R}
^{n}\right) $, we can find a sequence $\left( \alpha _{n}\right) _{n\in 
\mathbb{N}
}\subset C^{\infty }\left( 
\mathbb{R}
^{n}\right) \cap W_{\vartheta }^{1,p\left( .\right) }\left( 
\mathbb{R}
^{n}\right) $ such that $\alpha _{n}\longrightarrow f$ in $W_{\vartheta
}^{1,p\left( .\right) }\left( 
\mathbb{R}
^{n}\right) .$ Now, we take an open bounded neighborhood $U$ of $K$ such
that $f=1$ in $U.$ Also, we characterise a function $\alpha \in C^{\infty
}\left( 
\mathbb{R}
^{n}\right) ,$ $0\leq \alpha \leq 1$ be such that $\alpha =1$\ in $%
\mathbb{R}
^{n}-U$ and $\alpha =0$ in an open neighborhood of $K$. Then, $f$ or $\alpha 
$ is equal to one in $%
\mathbb{R}
^{n}.$ Now we define $\beta _{n}=1-\left( 1-\alpha _{n}\right) \alpha .$
Thus, we get 
\begin{eqnarray*}
f-\beta _{n} &=&\left( f-\alpha _{n}\right) \alpha +\left( 1-\alpha \right)
\left( f-1\right) \\
&=&\left( f-\alpha _{n}\right) \alpha .
\end{eqnarray*}%
Therefore, $\beta _{n}\longrightarrow f$ in $W_{\vartheta }^{1,p\left(
.\right) }\left( 
\mathbb{R}
^{n}\right) .$ Indeed, first, if we use the definitions of defined
functions, then we get%
\begin{eqnarray*}
\rho _{p\left( .\right) ,\vartheta }\left( \left( f-\alpha _{n}\right)
\alpha \right) &=&\dint\limits_{%
\mathbb{R}
^{n}-U}\left\vert f\left( x\right) -\alpha _{n}\left( x\right) \right\vert
^{p\left( x\right) }\vartheta \left( x\right) dx \\
&\leq &\rho _{p\left( .\right) ,\vartheta }\left( f-\alpha _{n}\right)
\longrightarrow 0.
\end{eqnarray*}%
Similarly, we have $\rho _{p\left( .\right) ,\vartheta }\left( \left\vert
\bigtriangledown \left( f-\alpha _{n}\right) \right\vert \right)
\longrightarrow 0.$Since $p^{+}<\infty ,$ we find that%
\begin{eqnarray*}
\left\Vert f-\beta _{n}\right\Vert _{1,p\left( .\right) ,\vartheta }
&=&\left\Vert \left( f-\alpha _{n}\right) \alpha \right\Vert _{1,p\left(
.\right) ,\vartheta } \\
&=&\left\Vert \left( f-\alpha _{n}\right) \alpha \right\Vert _{p\left(
.\right) ,\vartheta }+\left\Vert \bigtriangledown \left( \left( f-\alpha
_{n}\right) \alpha \right) \right\Vert _{p\left( .\right) ,\vartheta
}\longrightarrow 0.
\end{eqnarray*}%
Finally, since $\beta _{n}=1-\left( 1-\alpha _{n}\right) \alpha \in
S_{p\left( .\right) ,\vartheta }^{\infty }(K),$ it is clear to say that $%
S_{p\left( .\right) ,\vartheta }^{\infty }(K)$ is dense in $S_{p\left(
.\right) ,\vartheta }(K).$ This completes the proof.
\end{proof}

As in the proof [\cite{Dien4}, Proposition 10.1.10], we can show the
following theorem.

\begin{theorem}
Let $A\subset 
\mathbb{R}
^{n}$ and $1<p^{-}\leq p\left( .\right) \leq p^{+}<\infty ,$ $1<q^{-}\leq
q\left( .\right) \leq q^{+}<\infty $ with $q\left( .\right) \leq p\left(
.\right) .$ If $C_{p\left( .\right) ,\vartheta }\left( A\right) =0,$ then $%
C_{q\left( .\right) ,\vartheta }\left( A\right) =0.$
\end{theorem}

Now, we will introduce relative $\left( p\left( .\right) ,\vartheta \right) $%
- capacity.

\begin{definition}
Let $p\left( .\right) \in \mathcal{P}\left( \Omega \right) $ and $K\subset
\Omega $ be a compact subset. We denote%
\begin{equation*}
R_{p\left( .\right) ,\vartheta }\left( K,\Omega \right) =\left\{ f\in
W_{\vartheta }^{1,p(.)}\left( \Omega \right) \cap C_{0}\left( \Omega \right)
:f>1\text{ on }K\text{ and }f\geq 0\right\} ,
\end{equation*}%
set 
\begin{eqnarray*}
\text{cap}_{p\left( .\right) ,\vartheta }^{\ast }\left( K,\Omega \right)
&=&\inf_{f\in R_{p\left( .\right) ,\vartheta }\left( K,\Omega \right) }\rho
_{p\left( .\right) ,\vartheta }\left( \left\vert \bigtriangledown
f\right\vert \right) \\
&=&\inf_{f\in R_{p\left( .\right) ,\vartheta }\left( K,\Omega \right)
}\dint\limits_{\Omega }\left\vert \bigtriangledown f\left( x\right)
\right\vert ^{p\left( x\right) }\vartheta \left( x\right) dx.
\end{eqnarray*}%
\ \ Moreover, if $U\subset \Omega $ is an open subset, then we define%
\begin{equation*}
\text{cap}_{p\left( .\right) ,\vartheta }\left( U,\Omega \right) =\sup 
_{\substack{ K\subset U  \\ compact}}\text{cap}_{p\left( .\right) ,\vartheta
}^{\ast }\left( K,\Omega \right) ,
\end{equation*}%
and also for an arbitrary set $A\subset \Omega $ we define%
\begin{equation*}
\text{cap}_{p\left( .\right) ,\vartheta }\left( A,\Omega \right) =\inf 
_{\substack{ A\subset U\subset \Omega  \\ U\text{ }open}}\text{cap}_{p\left(
.\right) ,\vartheta }\left( U,\Omega \right) .
\end{equation*}%
We call cap$_{p\left( .\right) ,\vartheta }\left( A,\Omega \right) $ the
variational $\left( p\left( .\right) ,\vartheta \right) $-capacity of $A$
with respect to $\Omega .$ We say simply cap$_{p\left( .\right) ,\vartheta
}\left( A,\Omega \right) $ the relative $\left( p\left( .\right) ,\vartheta
\right) $- capacity. It is evident that the same number cap$_{p\left(
.\right) ,\vartheta }\left( A,\Omega \right) $ is obtained if the infimum in
the definition is taken over $f\in R_{p\left( .\right) ,\vartheta }\left(
K,\Omega \right) $ with $0\leq f\leq 1;$ when suitable, we implicitly assume
this extra condition.
\end{definition}

By the same arguments as in [\cite{Dien4}, Proposition 10.2.2] and [\cite%
{Dien4}, Proposition 10.2.3], we obtain Theorem \ref{atif1} and Theorem \ref%
{atif2}, respectively.

\begin{theorem}
\label{atif1}Let $K\subset \Omega $ be a compact subset. We denote 
\begin{equation*}
R_{p\left( .\right) ,\vartheta }^{\ast }\left( K,\Omega \right) =\left\{
f\in W_{\vartheta }^{1,p(.)}\left( \Omega \right) \cap C_{0}\left( \Omega
\right) :f\geq 1\text{ on }K\right\} .
\end{equation*}%
Then%
\begin{equation*}
\text{cap}_{p\left( .\right) ,\vartheta }^{\ast }\left( K,\Omega \right)
=\inf_{f\in R_{p\left( .\right) ,\vartheta }^{\ast }\left( K,\Omega \right)
}\rho _{p\left( .\right) ,\vartheta }\left( \left\vert \bigtriangledown
f\right\vert \right) .
\end{equation*}
\end{theorem}

\begin{theorem}
\label{atif2}Let $p\left( .\right) \in \mathcal{P}\left( \Omega \right) $
and $\vartheta $ is a weight function. Then, we have cap$_{p\left( .\right)
,\vartheta }^{\ast }\left( K,\Omega \right) =$cap$_{p\left( .\right)
,\vartheta }\left( K,\Omega \right) $ for every compact set $K\subset \Omega 
$.
\end{theorem}

Therefore the relative $\left( p\left( .\right) ,\vartheta \right) $-
capacity is well defined on compact sets. But, if $p^{+}=\infty ,$ then the
elements of the $R_{p\left( .\right) ,\vartheta }^{\ast }\left( K,\Omega
\right) $ do not satisfy equality in general. Also, the relative $\left(
p\left( .\right) ,\vartheta \right) $- capacity has the following properties.

\begin{enumerate}
\item[P1] . cap$_{p\left( .\right) ,\vartheta }\left( \emptyset ,\Omega
\right) =0.$

\item[P2] . If $A_{1}\subset A_{2}\subset \Omega _{2}\subset \Omega _{1},$
then cap$_{p\left( .\right) ,\vartheta }\left( A_{1},\Omega _{1}\right) \leq 
$cap$_{p\left( .\right) ,\vartheta }\left( A_{2},\Omega _{2}\right) .$

\item[P3] . If $A$ is a subset of $\Omega ,$ then%
\begin{equation*}
\text{cap}_{p\left( .\right) ,\vartheta }\left( A,\Omega \right) =\inf 
_{\substack{ A\subset U\subset \Omega  \\ U\text{ }open}}\text{cap}_{p\left(
.\right) ,\vartheta }\left( U,\Omega \right) .
\end{equation*}

\item[P4] . If $K_{1}$ and $K_{2}$ are compact subsets of $\Omega ,$ then%
\begin{eqnarray*}
\text{cap}_{p\left( .\right) ,\vartheta }\left( K_{1}\cup K_{2},\Omega
\right) +\text{cap}_{p\left( .\right) ,\vartheta }\left( K_{1}\cap
K_{2},\Omega \right) &\leq &\text{cap}_{p\left( .\right) ,\vartheta }\left(
K_{1},\Omega \right) \\
&&+\text{cap}_{p\left( .\right) ,\vartheta }\left( K_{2},\Omega \right) .
\end{eqnarray*}

\item[P5] . Let $K_{n}$ is a decreasing sequence of compact subsets of $%
\Omega $ for $n\in 
\mathbb{N}
.$ Then%
\begin{equation*}
\lim_{n\longrightarrow \infty }\text{cap}_{p\left( .\right) ,\vartheta
}\left( K_{n},\Omega \right) =\text{cap}_{p\left( .\right) ,\vartheta
}\left( \tbigcap\limits_{n=1}^{\infty }K_{n},\Omega \right) .
\end{equation*}

\item[P6] . If $A_{n}$ is an increasing sequence of subsets of $\Omega $ for 
$n\in 
\mathbb{N}
,$ then%
\begin{equation*}
\lim_{n\longrightarrow \infty }\text{cap}_{p\left( .\right) ,\vartheta
}\left( A_{n},\Omega \right) =\text{cap}_{p\left( .\right) ,\vartheta
}\left( \tbigcup\limits_{n=1}^{\infty }A_{n},\Omega \right) .
\end{equation*}

\item[P7] . If $A_{n}\subset \Omega $ for $n\in 
\mathbb{N}
,$ then%
\begin{equation*}
\text{cap}_{p\left( .\right) ,\vartheta }\left(
\tbigcup\limits_{n=1}^{\infty }A_{n},\Omega \right) \leq
\tsum\limits_{n=1}^{\infty }\text{cap}_{p\left( .\right) ,\vartheta }\left(
A_{n},\Omega \right) .
\end{equation*}
\end{enumerate}

The proof of these properties is the same as in \cite{Dien4}, \cite{Har1}, 
\cite{Hei}. Hence the relative $\left( p\left( .\right) ,\vartheta \right) $%
- capacity is an outer measure. A set function which satisfies the capacity
properties (P1), (P2), (P5) and (P6) is called Choquet capacity, see \cite%
{Cho}. Therefore we have the following result.

\begin{corollary}
\label{Chouquet}The set function $A\longrightarrow $cap$_{p\left( .\right)
,\vartheta }\left( A,\Omega \right) ,$ $A\subset \Omega ,$ is a Choquet
capacity. In particular, all Borel sets $A\subset \Omega $ are capacitable,
that is,%
\begin{equation*}
\text{cap}_{p\left( .\right) ,\vartheta }\left( A,\Omega \right) =\inf 
_{\substack{ A\subset U\subset \Omega  \\ U\text{ }open}}\text{cap}_{p\left(
.\right) ,\vartheta }\left( U,\Omega \right) =\sup_{\substack{ K\subset A 
\\ compact}}\text{cap}_{p\left( .\right) ,\vartheta }\left( K,\Omega \right)
.
\end{equation*}
\end{corollary}

Note that each Borel set is a Suslin set and the definition of Suslin sets
can be reach in \cite{Fed}. Also, it is not necessary that $p^{+}<\infty $
for satisfying all these properties.

\begin{theorem}
If $A_{1}\subset \Omega _{1}\subset A_{2}\subset \Omega _{2}\subset
...\subset \Omega =\tbigcup\limits_{n=1}^{\infty }\Omega _{n}$, then%
\begin{equation*}
\text{cap}_{p\left( .\right) ,\vartheta }\left( A_{1},\Omega \right) \leq
\left( \tsum\limits_{n=1}^{\infty }\left( \text{cap}_{p\left( .\right)
,\vartheta }\left( A_{n},\Omega _{n}\right) \right) ^{\frac{1}{1-p^{-}}%
}\right) ^{1-p^{-}}.
\end{equation*}
\end{theorem}

\begin{proof}
First we can assume that cap$_{p\left( .\right) ,\vartheta }\left(
A_{1},\Omega _{1}\right) <\infty $. Otherwise the proof is clear. Fix an
integer $m$. Also, let $\varepsilon >0$ and take an open set $U\subset
\Omega _{1}$ such that $A_{1}\subset U$ and%
\begin{equation}
\text{cap}_{p\left( .\right) ,\vartheta }\left( U,\Omega _{1}\right) \leq 
\text{cap}_{p\left( .\right) ,\vartheta }\left( A_{1},\Omega _{1}\right)
+\varepsilon .  \label{5}
\end{equation}%
Let $K_{1}\subset U$ be compact and let $f_{1}\in R_{p\left( .\right)
,\vartheta }\left( K_{1},\Omega _{1}\right) $ such that%
\begin{equation*}
\tint\limits_{\Omega _{1}}\left\vert \bigtriangledown f_{1}\left( x\right)
\right\vert ^{p\left( x\right) }\vartheta \left( x\right) dx\leq \text{cap}%
_{p\left( .\right) ,\vartheta }\left( K_{1},\Omega _{1}\right) +\varepsilon .
\end{equation*}%
Also, we can choose $f_{n}\in W_{\vartheta }^{1,p(.)}\left( \Omega \right)
\cap C_{0}\left( \Omega \right) ,$ $n=2,3,...,m$ such that $f_{n}\in
R_{p\left( .\right) ,\vartheta }\left( K_{n},\Omega _{n}\right) ,$ where $%
K_{n}=$supp$f_{n-1}$, and that%
\begin{equation*}
\tint\limits_{\Omega _{n}}\left\vert \bigtriangledown f_{n}\left( x\right)
\right\vert ^{p\left( x\right) }\vartheta \left( x\right) dx\leq \text{cap}%
_{p\left( .\right) ,\vartheta }\left( K_{n},\Omega _{n}\right) +\varepsilon
\end{equation*}%
by induction. Let $a_{n}$ be a sequence of nonnegative numbers with $%
\tsum\limits_{n=1}^{m}a_{n}=1$ and define $g=\tsum%
\limits_{n=1}^{m}a_{n}f_{n}.$ Since the space $W_{\vartheta }^{1,p(.)}\left(
\Omega \right) \cap C_{0}\left( \Omega \right) $ is a vector space, $g\in
W_{\vartheta }^{1,p(.)}\left( \Omega \right) \cap C_{0}\left( \Omega \right) 
$ and then $g\in R_{p\left( .\right) ,\vartheta }\left( K_{1},\Omega \right)
.$ It is easy to see that $K_{n}\subset \Omega _{n-1}\subset A_{n},$ $n\geq
2.$ Using the definition of relative $\left( p\left( .\right) ,\vartheta
\right) -$capacity, we have%
\begin{eqnarray*}
\text{cap}_{p\left( .\right) ,\vartheta }\left( K_{1},\Omega \right) &\leq
&\tint\limits_{\Omega _{1}\cup \Omega _{2}\cup ...}\left\vert
\tsum\limits_{n=1}^{m}a_{n}\bigtriangledown f_{n}\left( x\right) \right\vert
^{p\left( x\right) }\vartheta \left( x\right) dx \\
&\leq &\tsum\limits_{n=1}^{m}a_{n}^{p^{-}}\tint\limits_{\Omega
_{n}}\left\vert \bigtriangledown f_{n}\left( x\right) \right\vert ^{p\left(
x\right) }\vartheta \left( x\right) dx
\end{eqnarray*}%
where $\bigtriangledown f_{n}\neq 0$ are pairwise disjoint. This yields%
\begin{equation*}
\text{cap}_{p\left( .\right) ,\vartheta }\left( K_{1},\Omega \right) \leq
\left( a_{1}^{p^{-}}\text{cap}_{p\left( .\right) ,\vartheta }\left(
K_{1},\Omega _{1}\right) +\tsum\limits_{n=2}^{m}a_{n}^{p^{-}}\text{cap}%
_{p\left( .\right) ,\vartheta }\left( K_{n},\Omega _{n}\right) \right)
+\varepsilon ^{\ast }
\end{equation*}%
where $\varepsilon ^{\ast }=\varepsilon \tsum\limits_{n=1}^{m}a_{n}^{p^{-}}.$
Since $K_{1}\subset U,$ we get cap$_{p\left( .\right) ,\vartheta }\left(
K_{1},\Omega _{1}\right) \leq $cap$_{p\left( .\right) ,\vartheta }\left(
U,\Omega _{1}\right) .$ Also, it follows by the definition of relative $%
\left( p\left( .\right) ,\vartheta \right) -$capacity that cap$_{p\left(
.\right) ,\vartheta }\left( K_{n},\Omega _{n}\right) \leq $cap$_{p\left(
.\right) ,\vartheta }\left( A_{n},\Omega _{n}\right) $ and then $%
\tsum\limits_{n=2}^{m}a_{n}^{p^{-}}$cap$_{p\left( .\right) ,\vartheta
}\left( K_{n},\Omega _{n}\right) \leq \tsum\limits_{n=2}^{m}a_{n}^{p^{-}}$cap%
$_{p\left( .\right) ,\vartheta }\left( A_{n},\Omega _{n}\right) .$ Hence%
\begin{equation}
\text{cap}_{p\left( .\right) ,\vartheta }\left( K_{1},\Omega \right) \leq
a_{1}^{p^{-}}\text{cap}_{p\left( .\right) ,\vartheta }\left( U,\Omega
_{1}\right) +\tsum\limits_{n=2}^{m}a_{n}^{p^{-}}\text{cap}_{p\left( .\right)
,\vartheta }\left( A_{n},\Omega _{n}\right) +\varepsilon ^{\ast }.  \label{7}
\end{equation}%
If we use (\ref{5}) in (\ref{7}), then we have%
\begin{eqnarray*}
\text{cap}_{p\left( .\right) ,\vartheta }\left( K_{1},\Omega \right) &\leq
&\left( a_{1}^{p^{-}}\text{cap}_{p\left( .\right) ,\vartheta }\left(
A_{1},\Omega _{1}\right) \right. \\
&&\left. +\tsum\limits_{n=2}^{m}a_{n}^{p^{-}}\text{cap}_{p\left( .\right)
,\vartheta }\left( A_{n},\Omega _{n}\right) +\varepsilon ^{\ast
}a_{1}^{p^{-}}\right) +\varepsilon ^{\ast } \\
&=&\tsum\limits_{n=1}^{m}a_{n}^{p^{-}}\text{cap}_{p\left( .\right)
,\vartheta }\left( A_{n},\Omega _{n}\right) +\varepsilon ^{\ast \ast }
\end{eqnarray*}%
where $\varepsilon ^{\ast \ast }=\left( 1+a_{1}^{p^{-}}\right) \varepsilon
^{\ast }.$ Letting $\varepsilon ^{\ast \ast }\longrightarrow 0$ we get%
\begin{equation*}
\text{cap}_{p\left( .\right) ,\vartheta }\left( K_{1},\Omega \right) \leq
\tsum\limits_{n=1}^{m}a_{n}^{p^{-}}\text{cap}_{p\left( .\right) ,\vartheta
}\left( A_{n},\Omega _{n}\right) .
\end{equation*}%
Using the definition of infimum and relative $\left( p\left( .\right)
,\vartheta \right) -$ capacity, respectively, then we obtain%
\begin{equation}
\text{cap}_{p\left( .\right) ,\vartheta }\left( A_{1},\Omega \right) \leq 
\text{cap}_{p\left( .\right) ,\vartheta }\left( U,\Omega \right) \leq
\tsum\limits_{n=1}^{m}a_{n}^{p^{-}}\text{cap}_{p\left( .\right) ,\vartheta
}\left( A_{n},\Omega _{n}\right) .  \label{8}
\end{equation}%
Since the equality 
\begin{equation*}
\tsum\limits_{n=1}^{m}\left[ \text{cap}_{p\left( .\right) ,\vartheta }\left(
A_{n},\Omega _{n}\right) ^{\frac{1}{1-p^{-}}}\left( \tsum\limits_{k=1}^{m}%
\text{cap}_{p\left( .\right) ,\vartheta }\left( A_{k},\Omega _{k}\right) ^{%
\frac{1}{1-p^{-}}}\right) ^{-1}\right] =1
\end{equation*}%
holds, we can choose $a_{n}=$cap$_{p\left( .\right) ,\vartheta }\left(
A_{n},\Omega _{n}\right) ^{\frac{1}{1-p^{-}}}\left( \tsum\limits_{k=1}^{m}%
\text{cap}_{p\left( .\right) ,\vartheta }\left( A_{k},\Omega _{k}\right) ^{%
\frac{1}{1-p^{-}}}\right) ^{-1}$ for $n=1,2,..,m.$If cap$_{p\left( .\right)
,\vartheta }\left( A_{n},\Omega _{n}\right) >0$ for every $n=1,2,..,m,$ then
we have%
\begin{eqnarray*}
&&\text{cap}_{p\left( .\right) ,\vartheta }\left( A_{1},\Omega \right) \\
&\leq &\tsum\limits_{n=1}^{m}\text{cap}_{p\left( .\right) ,\vartheta }\left(
A_{n},\Omega _{n}\right) ^{1+\frac{p^{-}}{1-p^{-}}}\left(
\tsum\limits_{k=1}^{m}\text{cap}_{p\left( .\right) ,\vartheta }\left(
A_{k},\Omega _{k}\right) ^{\frac{1}{1-p^{-}}}\right) ^{-p^{-}} \\
&=&\left( \tsum\limits_{n=1}^{m}\text{cap}_{p\left( .\right) ,\vartheta
}\left( A_{n},\Omega _{n}\right) ^{\frac{1}{1-p^{-}}}\right) ^{1-p^{-}}.
\end{eqnarray*}%
When cap$_{p\left( .\right) ,\vartheta }\left( A_{n},\Omega _{n}\right) =0$
for some $n$, then cap$_{p\left( .\right) ,\vartheta }\left( A_{1},\Omega
\right) =0$ as well by considering (\ref{8}), and the proof is obvious. The
claim follows by letting $m\longrightarrow \infty .$
\end{proof}

\begin{remark}
\label{Liuembed}Let $\Omega \subset 
\mathbb{R}
^{n}$ be a bounded set. Then, the claim of Proposition 2.4 in \cite{Liu}
satisfies even if $p\left( .\right) =1.$ This yields $L_{\vartheta
}^{p\left( .\right) }\left( \Omega \right) \hookrightarrow L_{\vartheta
}^{1}\left( \Omega \right) $.
\end{remark}

\begin{theorem}
\label{eskap2}If cap$_{p\left( .\right) ,\vartheta }\left( B\left(
x_{0},r\right) ,B\left( x_{0},2r\right) \right) \geq 1$ and $\mu _{\vartheta
}$ is a doubling measure, then we obtain%
\begin{equation*}
C_{1}\mu _{\vartheta }\left( B\left( x_{0},r\right) \right) \leq \text{cap}%
_{p\left( .\right) ,\vartheta }\left( B\left( x_{0},r\right) ,B\left(
x_{0},2r\right) \right) \leq C_{2}\mu _{\vartheta }\left( B\left(
x_{0},r\right) \right)
\end{equation*}%
such that $C_{1}=\frac{C}{r}$ and $C_{2}=2^{p^{+}}c_{d}\max \left\{
r^{-p^{-}},r^{-p^{+}}\right\} .$
\end{theorem}

\begin{proof}
Let $f\in C_{0}^{\infty }\left( B\left( x_{0},2r\right) \right) $ is a
function such that $f=1$ in $B\left( x_{0},r\right) $ and $\left\vert
\bigtriangledown f\right\vert \leq \frac{2}{r}.$ Since $\mu _{\vartheta }$
is doubling we get%
\begin{eqnarray}
\text{cap}_{p\left( .\right) ,\vartheta }\left( B\left( x_{0},r\right)
,B\left( x_{0},2r\right) \right) &\leq &\dint\limits_{B\left(
x_{0},2r\right) }\left\vert \bigtriangledown f\left( x\right) \right\vert
^{p\left( x\right) }\vartheta \left( x\right) dx  \notag \\
&\leq &2^{p^{+}}c_{d}\max \left\{ r^{-p^{-}},r^{-p^{+}}\right\} \mu
_{\vartheta }\left( B\left( x_{0},r\right) \right) .  \label{doub4}
\end{eqnarray}%
On the other hand, let $0<s<r$ and take a function $f\in R_{p\left( .\right)
,\vartheta }^{\ast }\left( B\left( x_{0},s\right) ,B\left( x_{0},2r\right)
\right) $. Since cap$_{p\left( .\right) ,\vartheta }\left( B\left(
x_{0},r\right) ,B\left( x_{0},2r\right) \right) \geq 1,$ it is easy to see
that $\rho _{L_{\vartheta }^{p\left( .\right) }\left( B\left(
x_{0},2r\right) \right) }\left( \left\vert \bigtriangledown f\right\vert
\right) \geq 1$ and then we have $\left\Vert \bigtriangledown f\right\Vert
_{L_{\vartheta }^{p\left( .\right) }\left( B\left( x_{0},2r\right) \right)
}<\rho _{L_{\vartheta }^{p\left( .\right) }\left( B\left( x_{0},2r\right)
\right) }\left( \left\vert \bigtriangledown f\right\vert \right) ,$ see \cite%
{Liu}. Hence if we use the Poincar\'{e} inequality in $L_{\vartheta
}^{1}\left( B\left( x_{0},2r\right) \right) $ and the embedding $%
L_{\vartheta }^{p\left( .\right) }\left( B\left( x_{0},2r\right) \right)
\hookrightarrow L_{\vartheta }^{1}\left( B\left( x_{0},2r\right) \right) $,
then we obtain 
\begin{eqnarray}
\mu _{\vartheta }\left( B\left( x_{0},s\right) \right) &\leq
&cr\dint\limits_{B\left( x_{0},2r\right) }\left\vert \bigtriangledown
f\left( x\right) \right\vert \vartheta \left( x\right) dx\leq
crc_{1}\left\Vert \bigtriangledown f\right\Vert _{L_{\vartheta }^{p\left(
.\right) }\left( B\left( x_{0},2r\right) \right) }  \notag \\
&\leq &Cr\dint\limits_{B\left( x_{0},2r\right) }\left\vert \bigtriangledown
f\left( x\right) \right\vert ^{p\left( x\right) }\vartheta \left( x\right)
dx.  \label{doub2}
\end{eqnarray}%
If we take the infimum over $f\in R_{p\left( .\right) ,\vartheta }^{\ast
}\left( B\left( x_{0},s\right) ,B\left( x_{0},2r\right) \right) $ and
letting $s\rightarrow r$ from the inequality (\ref{doub2}), then we get%
\begin{equation}
\mu _{\vartheta }\left( B\left( x_{0},r\right) \right) \leq Cr\text{cap}%
_{p\left( .\right) ,\vartheta }\left( B\left( x_{0},r\right) ,B\left(
x_{0},2r\right) \right) .  \label{doub1}
\end{equation}%
We conclude the proof considering the inequalities (\ref{doub4}) and (\ref%
{doub1})$.$ Hence it is clear that we can write $\mu _{\vartheta }\left(
B\left( x_{0},r\right) \right) \approx $cap$_{p\left( .\right) ,\vartheta
}\left( B\left( x_{0},r\right) ,B\left( x_{0},2r\right) \right) $ under the
hypotheses.
\end{proof}

\begin{remark}
\label{Trick}Note that the equivalence in Theorem \ref{eskap2} is not true
in general. But if we use the following trick in inequality (\ref{doub2})%
\begin{eqnarray*}
\mu _{\vartheta }\left( B\left( x_{0},s\right) \right) &\leq
&cr\dint\limits_{B\left( x_{0},2r\right) }\left\vert \bigtriangledown
f\left( x\right) \right\vert \vartheta \left( x\right) dx \\
&\leq &cr\dint\limits_{B\left( x_{0},2r\right) }\max \left\{ 1,\left\vert
\bigtriangledown f\left( x\right) \right\vert \right\} ^{p\left( x\right)
}\vartheta \left( x\right) dx \\
&\leq &cr\dint\limits_{B\left( x_{0},2r\right) }\left( 1+\left\vert
\bigtriangledown f\left( x\right) \right\vert ^{p\left( x\right) }\right)
\vartheta \left( x\right) dx \\
&\leq &cr\left( \mu _{\vartheta }\left( B\left( x_{0},2r\right) \right)
+\rho _{p\left( .\right) ,\vartheta }\left( \left\vert \bigtriangledown
f\right\vert \right) \right) ,
\end{eqnarray*}%
then this will allow for obtaining some estimates even in case cap$_{p\left(
.\right) ,\vartheta }\left( B\left( x_{0},r\right) ,B\left( x_{0},2r\right)
\right) <1.$
\end{remark}

\begin{theorem}
\label{eskap1}If $A\subset B\left( x_{0},r\right) ,$ cap$_{p\left( .\right)
,\vartheta }\left( A,B\left( x_{0},4r\right) \right) \geq 1$ and $0<r\leq
s\leq 2r,$ then%
\begin{equation*}
\frac{1}{C}\text{cap}_{p\left( .\right) ,\vartheta }\left( A,B\left(
x_{0},2r\right) \right) \leq \text{cap}_{p\left( .\right) ,\vartheta }\left(
A,B\left( x_{0},2s\right) \right) \leq \text{cap}_{p\left( .\right)
,\vartheta }\left( A,B\left( x_{0},2r\right) \right)
\end{equation*}%
such that $C=2^{p^{+}}+2^{2p^{+}+1}cc_{1}\max \left\{
r^{1-p^{-}},r^{1-p^{+}}\right\} $.
\end{theorem}

\begin{proof}
Since $B\left( x_{0},2r\right) \subset B\left( x_{0},2s\right) ,$ it is
clear that%
\begin{equation*}
\text{cap}_{p\left( .\right) ,\vartheta }\left( A,B\left( x_{0},2s\right)
\right) \leq \text{cap}_{p\left( .\right) ,\vartheta }\left( A,B\left(
x_{0},2r\right) \right) .
\end{equation*}%
Thus, we need to satisfy the first inequality in case $s=2r.$ Because of the
fact that relative $\left( p\left( .\right) ,\vartheta \right) $- capacity
is a Choquet capacity, we can suppose that $A$ is compact. Let $g\in
C_{0}^{\infty }\left( B\left( x_{0},2r\right) \right) ,$ $0\leq g\leq 1$ is
a cut-off function such that $g=1$ in $B\left( x_{0},r\right) $ and $%
\left\vert \bigtriangledown g\right\vert \leq \frac{2}{r}.$ Also, let the
function $f\in R_{p\left( .\right) ,\vartheta }^{\ast }\left( A,B\left(
x_{0},4r\right) \right) $ be given. If we use the definition of $R_{p\left(
.\right) ,\vartheta }^{\ast }\left( A,B\left( x_{0},4r\right) \right) $ and
the function $g$ and also the fact that the space $C_{0}^{\infty }\left(
B\left( x_{0},2r\right) \right) $ is dense in $W_{\vartheta }^{1,p\left(
.\right) }\left( B\left( x_{0},2r\right) \right) ,$ then we get that $gf\in
W_{\vartheta }^{1,p(.)}\left( B\left( x_{0},2r\right) \right) \cap
C_{0}\left( B\left( x_{0},2r\right) \right) $ such that $gf=1$ on $A$. Thus $%
gf\in R_{p\left( .\right) ,\vartheta }^{\ast }\left( A,B\left(
x_{0},2r\right) \right) .$ Therefore, we have%
\begin{eqnarray*}
&&\text{cap}_{p\left( .\right) ,\vartheta }\left( A,B\left( x_{0},2r\right)
\right) \\
&\leq &2^{p^{+}}\dint\limits_{B\left( x_{0},2r\right) }\left\vert
\bigtriangledown f\left( x\right) \right\vert ^{p\left( x\right) }\vartheta
\left( x\right) dx \\
&&+2^{2p^{+}}\max \left\{ r^{-p^{-}},r^{-p^{+}}\right\}
\dint\limits_{B\left( x_{0},2r\right) }\left\vert f\left( x\right)
\right\vert ^{p\left( x\right) }\vartheta \left( x\right) dx \\
&\leq &2^{p^{+}}\dint\limits_{B\left( x_{0},4r\right) }\left\vert
\bigtriangledown f\left( x\right) \right\vert ^{p\left( x\right) }\vartheta
\left( x\right) dx \\
&&+2^{2p^{+}}\max \left\{ r^{-p^{-}},r^{-p^{+}}\right\}
\dint\limits_{B\left( x_{0},4r\right) }\left\vert f\left( x\right)
\right\vert ^{p\left( x\right) }\vartheta \left( x\right) dx.
\end{eqnarray*}%
Since cap$_{p\left( .\right) ,\vartheta }\left( A,B\left( x_{0},4r\right)
\right) \geq 1,$ we have $\left\Vert \bigtriangledown f\right\Vert
_{L_{\vartheta }^{p\left( .\right) }\left( B\left( x_{0},4r\right) \right)
}<\rho _{L_{\vartheta }^{p\left( .\right) }\left( B\left( x_{0},4r\right)
\right) }\left( \left\vert \bigtriangledown f\right\vert \right) ,$ see \cite%
{Liu}. Hence if we use the Poincar\'{e} inequality in $L_{\vartheta
}^{1}\left( B\left( x_{0},4r\right) \right) $ and the embedding $%
L_{\vartheta }^{p\left( .\right) }\left( B\left( x_{0},4r\right) \right)
\hookrightarrow L_{\vartheta }^{1}\left( B\left( x_{0},4r\right) \right) $,
then we obtain%
\begin{eqnarray*}
\dint\limits_{B\left( x_{0},4r\right) }\left\vert f\left( x\right)
\right\vert ^{p\left( x\right) }\vartheta \left( x\right) dx &\leq
&2rc\dint\limits_{B\left( x_{0},4r\right) }\left\vert \bigtriangledown
f\left( x\right) \right\vert \vartheta \left( x\right) dx \\
&\leq &2rcc_{1}\left\Vert \bigtriangledown f\right\Vert _{L_{\vartheta
}^{p\left( .\right) }\left( B\left( x_{0},4r\right) \right) } \\
&\leq &2rcc_{1}\rho _{L_{\vartheta }^{p\left( .\right) }\left( B\left(
x_{0},4r\right) \right) }\left( \left\vert \bigtriangledown f\right\vert
\right) .
\end{eqnarray*}%
This yields%
\begin{equation*}
\text{cap}_{p\left( .\right) ,\vartheta }\left( A,B\left( x_{0},2r\right)
\right) \leq C\dint\limits_{B\left( x_{0},4r\right) }\left\vert
\bigtriangledown f\left( x\right) \right\vert ^{p\left( x\right) }\vartheta
\left( x\right) dx
\end{equation*}%
where $C=2^{p^{+}}+2^{2p^{+}+1}cc_{1}\max \left\{
r^{1-p^{-}},r^{1-p^{+}}\right\} .$ The proof is completed by taking the
infimum over $f\in R_{p\left( .\right) ,\vartheta }^{\ast }\left( A,B\left(
x_{0},4r\right) \right) $ from the last inequality. Hence it is clear that
we can write cap$_{p\left( .\right) ,\vartheta }\left( A,B\left(
x_{0},2s\right) \right) \approx $cap$_{p\left( .\right) ,\vartheta }\left(
A,B\left( x_{0},2r\right) \right) $ under the hypotheses.
\end{proof}

\begin{remark}
By the same arguments as in Theorem \ref{eskap2} the equivalence in Theorem %
\ref{eskap1} is not true in general. But if we use the same trick in Remark %
\ref{Trick}, then it can be found some estimates even in case cap$_{p\left(
.\right) ,\vartheta }\left( A,B\left( x_{0},4r\right) \right) <1.$
\end{remark}

\begin{theorem}
Let $1<p^{-}\leq p\left( .\right) \leq p^{+}<\infty ,$ $1<q^{-}\leq q\left(
.\right) \leq q^{+}<\infty $ and $\frac{1}{p\left( .\right) }+\frac{1}{%
q\left( .\right) }=1.$ Assume that $\vartheta $ is a weight function such
that $\vartheta \left( x\right) \geq 1$ for $x\in 
\mathbb{R}
^{n}.$ If $0<r_{1}<r_{2}<\infty $ and cap$_{p\left( .\right) ,\vartheta
}\left( A\left( x_{0};r_{1},r_{2}\right) ,B\left( x_{0},r_{2}\right) \right)
\geq 1,$ then%
\begin{equation*}
\omega _{n-1}\leq C\text{cap}_{p\left( .\right) ,\vartheta }\left( B\left(
x_{0},r_{1}\right) ,B\left( x_{0},r_{2}\right) \right)
\end{equation*}%
where $A\left( x_{0};r_{1},r_{2}\right) $ is the annulus $B\left(
x_{0},r_{2}\right) -B\left( x_{0},r_{1}\right) .$ Here%
\begin{eqnarray*}
C &=&c_{h}\max \left\{ \left[ \max \left\{ r_{2}^{\left( 1-n\right)
q^{+}},r_{2}^{\left( 1-n\right) q^{-}}\right\} \left\vert A\left(
x_{0};r_{1},r_{2}\right) \right\vert \right] ^{\frac{1}{q^{+}}},\right. \\
&&\left. \left[ \max \left\{ r_{2}^{\left( 1-n\right) q^{+}},r_{2}^{\left(
1-n\right) q^{-}}\right\} \left\vert A\left( x_{0};r_{1},r_{2}\right)
\right\vert \right] ^{\frac{1}{q^{-}}}\right\} .
\end{eqnarray*}%
where $\left\vert A\left( x_{0};r_{1},r_{2}\right) \right\vert $ is the
Lebesgue measure of $A\left( x_{0};r_{1},r_{2}\right) $ and $c_{h}$ is the
constant of H\"{o}lder inequality for variable exponent Lebesgue spaces$.$
\end{theorem}

\begin{proof}
Let $f\in C_{0}^{\infty }\left( B\left( x_{0},r_{2}\right) \right) $ be a
function such that $f=1$ on $B\left( x_{0},r_{1}\right) .$ Then $f\in
R_{p\left( .\right) ,\vartheta }^{\ast }\left( B\left( x_{0},r_{1}\right)
,B\left( x_{0},r_{2}\right) \right) .$ By [\cite{Gil}, Lemma 7.14], we get%
\begin{equation*}
f\left( y\right) =\frac{1}{n\omega _{n}}\dint\limits_{%
\mathbb{R}
^{n}}\frac{\bigtriangledown f\left( x\right) \left( y-x\right) }{\left\vert
x-y\right\vert ^{n}}dx.
\end{equation*}%
Also, it is well known that $\left( n-1\right) -$ dimensional measure of the
unit sphere $\omega _{n-1}$ in $%
\mathbb{R}
^{n}$ equals $n\omega _{n}.$ Hence the following integral is obtained%
\begin{equation*}
f\left( y\right) =\frac{1}{\omega _{n-1}}\dint\limits_{%
\mathbb{R}
^{n}}\frac{\bigtriangledown f\left( x\right) \left( y-x\right) }{\left\vert
x-y\right\vert ^{n}}dx
\end{equation*}%
for all $y\in 
\mathbb{R}
^{n}$. Since cap$_{p\left( .\right) ,\vartheta }\left( A\left(
x_{0};r_{1},r_{2}\right) ,B\left( x_{0},r_{2}\right) \right) \geq 1,$ it is
easy to see that $\rho _{L_{\vartheta }^{p\left( .\right) }\left( B\left(
x_{0},r_{2}\right) \right) }\left( \left\vert \bigtriangledown f\right\vert
\right) \geq 1$ and then we have $\left\Vert \bigtriangledown f\right\Vert
_{L_{\vartheta }^{p\left( .\right) }\left( B\left( x_{0},r_{2}\right)
\right) }<\rho _{L_{\vartheta }^{p\left( .\right) }\left( B\left(
x_{0},r_{2}\right) \right) }\left( \left\vert \bigtriangledown f\right\vert
\right) ,$ see \cite{Liu}. Also, if we use the H\"{o}lder inequality for
variable exponent Lebesgue spaces, then we find%
\begin{eqnarray*}
\omega _{n-1} &=&\omega _{n-1}f\left( x_{0}\right) \\
&=&\dint\limits_{A\left( x_{0};r_{1},r_{2}\right) }\frac{\bigtriangledown
f\left( x\right) \left( x_{0}-x\right) }{\left\vert x-x_{0}\right\vert ^{n}}%
\vartheta \left( x\right) ^{\frac{1}{p\left( x\right) }}\vartheta \left(
x\right) ^{-\frac{1}{p\left( x\right) }}dx \\
&\leq &c_{h}\left\Vert \left\vert x-x_{0}\right\vert ^{1-n}\vartheta \left(
x\right) ^{-\frac{1}{p\left( .\right) }}\right\Vert _{L^{q\left( .\right)
}\left( A\left( x_{0};r_{1},r_{2}\right) \right) }\rho _{L_{\vartheta
}^{p\left( .\right) }\left( B\left( x_{0},r_{2}\right) \right) }\left(
\left\vert \bigtriangledown f\right\vert \right)
\end{eqnarray*}%
for some $c_{h}>0$ where $\frac{1}{p\left( .\right) }+\frac{1}{q\left(
.\right) }=1.$ Using the relationship between Luxemburg norm and modular
(see \cite{K}), we get%
\begin{eqnarray*}
&&\omega _{n-1} \\
&\leq &c_{h}\max \left\{ \left( \rho _{L^{q\left( .\right) }\left( A\left(
x_{0};r_{1},r_{2}\right) \right) }\left( \left\vert x-x_{0}\right\vert
^{1-n}\vartheta \left( x\right) ^{-\frac{1}{p\left( .\right) }}\right)
\right) ^{\frac{1}{q^{+}}},\right. \\
&&\left. \left( \rho _{L^{q\left( .\right) }\left( A\left(
x_{0};r_{1},r_{2}\right) \right) }\left( \left\vert x-x_{0}\right\vert
^{1-n}\vartheta \left( x\right) ^{-\frac{1}{p\left( .\right) }}\right)
\right) ^{\frac{1}{q^{-}}}\right\} \rho _{L_{\vartheta }^{p\left( .\right)
}\left( B\left( x_{0},r_{2}\right) \right) }\left( \left\vert
\bigtriangledown f\right\vert \right) \\
&\leq &c_{h}\max \left\{ \left[ \max \left\{ r_{2}^{\left( 1-n\right)
q^{+}},r_{2}^{\left( 1-n\right) q^{-}}\right\} \left\vert A\left(
x_{0};r_{1},r_{2}\right) \right\vert \right] ^{\frac{1}{q^{+}}},\right. \\
&&\left. \left[ \max \left\{ r_{2}^{\left( 1-n\right) q^{+}},r_{2}^{\left(
1-n\right) q^{-}}\right\} \left\vert A\left( x_{0};r_{1},r_{2}\right)
\right\vert \right] ^{\frac{1}{q^{-}}}\right\} \rho _{L_{\vartheta
}^{p\left( .\right) }\left( B\left( x_{0},r_{2}\right) \right) }\left(
\left\vert \bigtriangledown f\right\vert \right)
\end{eqnarray*}%
for some $c_{h}>0.$ Taking the infimum over $f\in R_{p\left( .\right)
,\vartheta }^{\ast }\left( B\left( x_{0},r_{1}\right) ,B\left(
x_{0},r_{2}\right) \right) $ from the last inequality, we have the desired
result by the continuity of the integral.
\end{proof}

\section{The Relationship Between Capacities}

Now, we will give several inequalities between the capacities previously
mentioned.

\begin{theorem}
\label{kar}If $\Omega \subset 
\mathbb{R}
^{n}$ is bounded and $K\subset \Omega $ is compact, then%
\begin{equation*}
C_{p\left( .\right) ,\vartheta }\left( K\right) \leq C\max \left\{ \text{cap}%
_{p\left( .\right) ,\vartheta }\left( K,\Omega \right) ^{\frac{1}{p^{+}}},%
\text{cap}_{p\left( .\right) ,\vartheta }\left( K,\Omega \right) \right\}
\end{equation*}%
where the constant $C$ depends on the dimension $n,$ the Poincar\'{e}
inequality constant and diam$\left( \Omega \right) .$
\end{theorem}

\begin{proof}
We can assume that cap$_{p\left( .\right) ,\vartheta }\left( K,\Omega
\right) <\infty $. Otherwise the proof is clear. Let $0<\varepsilon <1$ and $%
f\in R_{p\left( .\right) ,\vartheta }^{\ast }\left( K,\Omega \right) $ be a
function such that%
\begin{equation}
\rho _{p\left( .\right) ,\vartheta }\left( \left\vert \bigtriangledown
f\right\vert \right) \leq \text{cap}_{p\left( .\right) ,\vartheta }\left(
K,\Omega \right) +\varepsilon .  \label{13}
\end{equation}%
Now, let us extend $f$ by zero outside of $\Omega ,$ that is%
\begin{equation*}
f\left( x\right) =\left\{ 
\begin{array}{c}
f\left( x\right) ,\text{ \ \ }x\in \Omega \\ 
0,\text{ \ }x\in 
\mathbb{R}
^{n}-\Omega%
\end{array}%
\right. ,
\end{equation*}%
and define $g=\min \left\{ 1,f\right\} .$ If we consider definitions of the
relative $\left( p\left( .\right) ,\vartheta \right) -$ capacity and the
Sobolev capacity, then we get $g\in S_{p\left( .\right) ,\vartheta }\left(
K\right) .$ Hence%
\begin{equation*}
C_{p\left( .\right) ,\vartheta }\left( K\right) \leq \dint\limits_{%
\mathbb{R}
^{n}}\left( \left\vert g\left( x\right) \right\vert ^{p\left( x\right)
}+\left\vert \bigtriangledown f\left( x\right) \right\vert ^{p\left(
x\right) }\right) \vartheta \left( x\right) dx.
\end{equation*}%
It follows by $0\leq g\leq 1$ that%
\begin{equation}
\dint\limits_{\Omega }\left\vert g\left( x\right) \right\vert ^{p\left(
x\right) }\vartheta \left( x\right) dx\leq \dint\limits_{\Omega }\left\vert
g\left( x\right) \right\vert \vartheta \left( x\right) dx\leq
\dint\limits_{\Omega }\left\vert f\left( x\right) \right\vert \vartheta
\left( x\right) dx.  \label{9}
\end{equation}%
Also, if we use the Poincar\'{e} inequality in $L_{\vartheta }^{1}\left(
\Omega \right) $ and Remark \ref{Liuembed}, then we have%
\begin{equation}
\left\Vert f\vartheta \right\Vert _{1}=\left\Vert f\right\Vert _{1,\vartheta
}\leq c\text{diam}\left( \Omega \right) \left\Vert \bigtriangledown
f\right\Vert _{1,\vartheta }\leq c\text{diam}\left( \Omega \right)
c_{1}\left\Vert \bigtriangledown f\right\Vert _{p\left( .\right) ,\vartheta
}.  \label{10}
\end{equation}%
By (\ref{9}) and (\ref{10}), we have%
\begin{eqnarray*}
C_{p\left( .\right) ,\vartheta }\left( K\right) &\leq &\dint\limits_{\Omega
}\left\vert f\left( x\right) \right\vert \vartheta \left( x\right)
dx+\dint\limits_{\Omega }\left\vert \bigtriangledown f\left( x\right)
\right\vert ^{p\left( x\right) }\vartheta \left( x\right) dx \\
&\leq &C^{\ast }\left[ \left\Vert \bigtriangledown f\right\Vert _{p\left(
.\right) ,\vartheta }+\rho _{p\left( .\right) ,\vartheta }\left( \left\vert
\bigtriangledown f\right\vert \right) \right] \\
&\leq &C^{\ast }\left( \max \left\{ \rho _{p\left( .\right) ,\vartheta
}\left( \left\vert \bigtriangledown f\right\vert \right) ^{\frac{1}{p^{+}}%
},\rho _{p\left( .\right) ,\vartheta }\left( \left\vert \bigtriangledown
f\right\vert \right) ^{\frac{1}{p^{-}}}\right\} +\rho _{p\left( .\right)
,\vartheta }\left( \left\vert \bigtriangledown f\right\vert \right) \right)
\end{eqnarray*}%
where $C^{\ast }=\max \left\{ 1,c\text{diam}\left( \Omega \right)
c_{1}\right\} .$ Considering the fact that $1<p^{-}\leq p\left( .\right)
\leq p^{+}<\infty $ and (\ref{13})$,$ it is to see that%
\begin{eqnarray*}
&&\max \left\{ \rho _{p\left( .\right) ,\vartheta }\left( \left\vert
\bigtriangledown f\right\vert \right) ^{\frac{1}{p^{+}}},\rho _{p\left(
.\right) ,\vartheta }\left( \left\vert \bigtriangledown f\right\vert \right)
^{\frac{1}{p^{-}}}\right\} +\rho _{p\left( .\right) ,\vartheta }\left(
\left\vert \bigtriangledown f\right\vert \right) \\
&\leq &2\max \left\{ \rho _{p\left( .\right) ,\vartheta }\left( \left\vert
\bigtriangledown f\right\vert \right) ^{\frac{1}{p^{+}}},\rho _{p\left(
.\right) ,\vartheta }\left( \left\vert \bigtriangledown f\right\vert \right)
\right\} \\
&\leq &2\max \left\{ \left( \text{cap}_{p\left( .\right) ,\vartheta }\left(
K,\Omega \right) +\varepsilon \right) ^{\frac{1}{p^{+}}},\text{cap}_{p\left(
.\right) ,\vartheta }\left( K,\Omega \right) +\varepsilon \right\} \\
&\leq &2\max \left\{ \left( \text{cap}_{p\left( .\right) ,\vartheta }\left(
K,\Omega \right) \right) ^{\frac{1}{p^{+}}},\text{cap}_{p\left( .\right)
,\vartheta }\left( K,\Omega \right) \right\} +\varepsilon ^{\frac{1}{p^{+}}%
}+\varepsilon .
\end{eqnarray*}%
Hence, we get%
\begin{equation*}
C_{p\left( .\right) ,\vartheta }\left( K\right) \leq C\left[ \max \left\{
\left( \text{cap}_{p\left( .\right) ,\vartheta }\left( K,\Omega \right)
\right) ^{\frac{1}{p^{+}}},\text{cap}_{p\left( .\right) ,\vartheta }\left(
K,\Omega \right) \right\} +\varepsilon ^{\frac{1}{p^{+}}}+\varepsilon \right]
\end{equation*}%
where $C=2\max \left\{ 1,c\text{diam}\left( \Omega \right) c_{1}\right\} .$
This yields the claim as $\varepsilon $ tends to zero.
\end{proof}

The proof of the following theorem is similar to [\cite{Dien4}, Theorem
10.3.2].

\begin{theorem}
\label{Kap}If $\Omega \subset 
\mathbb{R}
^{n}$ is bounded and $A\subset \Omega $, then%
\begin{equation*}
C_{p\left( .\right) ,\vartheta }\left( A\right) \leq C\max \left\{ \text{cap}%
_{p\left( .\right) ,\vartheta }\left( A,\Omega \right) ^{\frac{1}{p^{+}}},%
\text{cap}_{p\left( .\right) ,\vartheta }\left( A,\Omega \right) \right\}
\end{equation*}%
where the constant $C$ depends on the dimension $n,$ the Poincar\'{e}
inequality constant and diam$\left( \Omega \right) .$
\end{theorem}

\begin{corollary}
Let $A\subset \Omega .$ If cap$_{p\left( .\right) ,\vartheta }\left(
A,\Omega \right) =0,$ then $C_{p\left( .\right) ,\vartheta }\left( A\right)
=0.$
\end{corollary}

Note that the opposite implication of previous corollary does not always
true. We need to consider an additional hypothesis for this. By the same
arguments as in [\cite{Dien4}, Proposition 10.3.4], we obtain following
statement.

\begin{theorem}
Let $A\subset \Omega $. Assume that the space $W_{\vartheta }^{1,p\left(
.\right) }\left( 
\mathbb{R}
^{n}\right) \cap C\left( 
\mathbb{R}
^{n}\right) $ is dense in $W_{\vartheta }^{1,p\left( .\right) }\left( 
\mathbb{R}
^{n}\right) .$ If $C_{p\left( .\right) ,\vartheta }\left( A\right) =0,$ then
cap$_{p\left( .\right) ,\vartheta }\left( A,\Omega \right) =0.$
\end{theorem}

Now, we give a relationship between Sobolev $\left( p\left( .\right)
,\vartheta \right) $- capacity and relative $\left( p\left( .\right)
,\vartheta \right) $- capacity.

\begin{theorem}
If $A\subset B\left( x_{0},r\right) $ and cap$_{p\left( .\right) ,\vartheta
}\left( A,B\left( x_{0},2r\right) \right) \geq 1,$ then%
\begin{equation*}
\frac{1}{C_{1}}C_{p\left( .\right) ,\vartheta }\left( A\right) \leq \text{cap%
}_{p\left( .\right) ,\vartheta }\left( A,B\left( x_{0},2r\right) \right)
\leq C_{2}C_{p\left( .\right) ,\vartheta }\left( A\right)
\end{equation*}%
where $C_{1}=$ $1+cr\left( 1+\left\vert B\left( x_{0},2r\right) \right\vert
\right) $ and $C_{2}=2^{2p^{+}}\left( 1+\max \left\{
r^{-p^{-}},r^{-p^{+}}\right\} \right) $ and $c$ is the Poincar\'{e}
inequality constant.
\end{theorem}

\begin{proof}
Suppose that $K\subset B\left( x_{0},r\right) $ is compact. Let $g\in
C_{0}^{\infty }\left( B\left( x_{0},2r\right) \right) ,$ $0\leq g\leq 1$ is
a cut-off function such that $g=1$ in $B\left( x_{0},r\right) $ and $%
\left\vert \bigtriangledown g\right\vert \leq \frac{2}{r}.$ Also, the
function $f\in S_{p\left( .\right) ,\vartheta }\left( K\right) $ be given.
Thus we get $gf\in R_{p\left( .\right) ,\vartheta }^{\ast }\left( A,B\left(
x_{0},2r\right) \right) .$ Therefore%
\begin{eqnarray*}
&&\text{cap}_{p\left( .\right) ,\vartheta }\left( K,B\left( x_{0},2r\right)
\right) \\
&\leq &2^{p^{+}}\dint\limits_{B\left( x_{0},2r\right) }\left\vert
\bigtriangledown f\left( x\right) \right\vert ^{p\left( x\right) }\vartheta
\left( x\right) dx \\
&&+2^{2p^{+}}\max \left\{ r^{-p^{-}},r^{-p^{+}}\right\}
\dint\limits_{B\left( x_{0},2r\right) }\left\vert f\left( x\right)
\right\vert ^{p\left( x\right) }\vartheta \left( x\right) dx \\
&\leq &2^{2p^{+}}\left( 1+\max \left\{ r^{-p^{-}},r^{-p^{+}}\right\} \right)
\dint\limits_{%
\mathbb{R}
^{n}}\left( \left\vert f\left( x\right) \right\vert ^{p\left( x\right)
}+\left\vert \bigtriangledown f\left( x\right) \right\vert ^{p\left(
x\right) }\right) \vartheta \left( x\right) dx.
\end{eqnarray*}%
If we take the infimum over $f\in S_{p\left( .\right) ,\vartheta }\left(
K\right) $ from the last inequality, then we have%
\begin{equation*}
\text{cap}_{p\left( .\right) ,\vartheta }\left( K,B\left( x_{0},2r\right)
\right) \leq C_{2}C_{p\left( .\right) ,\vartheta }\left( K\right)
\end{equation*}%
where $C_{2}=2^{2p^{+}}\left( 1+\max \left\{ r^{-p^{-}},r^{-p^{+}}\right\}
\right) .$

Now, we take $f\in C_{0}^{\infty }\left( B\left( x_{0},2r\right) \right) ,$ $%
0\leq f\leq 1$ such that $f=1$ in open set containing $K.$ Then $f\in
R_{p\left( .\right) ,\vartheta }^{\ast }\left( K,B\left( x_{0},2r\right)
\right) .$Since cap$_{p\left( .\right) ,\vartheta }\left( A,B\left(
x_{0},2r\right) \right) \geq 1,$ it is easy to see that $\rho _{L_{\vartheta
}^{p\left( .\right) }\left( B\left( x_{0},2r\right) \right) }\left(
\left\vert \bigtriangledown f\right\vert \right) \geq 1$ and then we have $%
\left\Vert \bigtriangledown f\right\Vert _{L_{\vartheta }^{p\left( .\right)
}\left( B\left( x_{0},2r\right) \right) }<\rho _{L_{\vartheta }^{p\left(
.\right) }\left( B\left( x_{0},2r\right) \right) }\left( \left\vert
\bigtriangledown f\right\vert \right) ,$ see \cite{Liu}. If we use the fact $%
0\leq f\leq 1,$ the Poincar\'{e} inequality in $L_{\vartheta }^{1}\left(
B\left( x_{0},2r\right) \right) $ and the embedding $L_{\vartheta }^{p\left(
.\right) }\left( B\left( x_{0},2r\right) \right) \hookrightarrow
L_{\vartheta }^{1}\left( B\left( x_{0},2r\right) \right) $ , then we obtain%
\begin{eqnarray*}
\dint\limits_{%
\mathbb{R}
^{n}}\left\vert f\left( x\right) \right\vert ^{p\left( x\right) }\vartheta
\left( x\right) dx &\leq &cr\dint\limits_{B\left( x_{0},2r\right)
}\left\vert \bigtriangledown f\left( x\right) \right\vert \vartheta \left(
x\right) dx \\
&\leq &crc_{1}\dint\limits_{B\left( x_{0},2r\right) }\left\vert
\bigtriangledown f\left( x\right) \right\vert ^{p\left( x\right) }\vartheta
\left( x\right) dx.
\end{eqnarray*}%
It follows that%
\begin{equation*}
C_{p\left( .\right) ,\vartheta }\left( K\right) \leq
C_{1}\dint\limits_{B\left( x_{0},2r\right) }\left\vert \bigtriangledown
f\left( x\right) \right\vert ^{p\left( x\right) }\vartheta \left( x\right) dx
\end{equation*}%
where $C_{1}=$ $1+crc_{1}.$ This completes the proof for the compact sets if
we take the infimum over $f\in R_{p\left( .\right) ,\vartheta }^{\ast
}\left( K,B\left( x_{0},2r\right) \right) $ from the last inequality. If we
consider the definition of relative $\left( p\left( .\right) ,\vartheta
\right) $- capacity and use the first part of proof, then it is shown that
the desired result holds for arbitrary set $A\subset B\left( x_{0},r\right)
. $
\end{proof}

\section{$\left( p\left( .\right) ,\protect\vartheta \right) $- Thinness}

\begin{definition}
The set $A\subset 
\mathbb{R}
^{n}$ is called $\left( p\left( .\right) ,\vartheta \right) $- thin at $%
x_{0} $ if%
\begin{equation}
\tint\limits_{0}^{1}\left( \frac{\text{cap}_{p\left( .\right) ,\vartheta
}\left( A\cap B\left( x_{0},r\right) ,B\left( x_{0},2r\right) \right) }{%
\text{cap}_{p\left( .\right) ,\vartheta }\left( B\left( x_{0},r\right)
,B\left( x_{0},2r\right) \right) }\right) ^{\frac{1}{p\left( x_{0}\right) -1}%
}\frac{dr}{r}<\infty .  \label{Wein}
\end{equation}%
We say that $A$ is $\left( p\left( .\right) ,\vartheta \right) $- thick at $%
x_{0}$ if $A$ is not $\left( p\left( .\right) ,\vartheta \right) $- thin at $%
x_{0}.$
\end{definition}

The integral in the inequality (\ref{Wein}) is called Wiener type integral,
see \cite{Hei}. From now on, we write that%
\begin{equation*}
W_{p\left( .\right) ,\vartheta }\left( A,x_{0}\right)
=\tint\limits_{0}^{1}\left( \frac{\text{cap}_{p\left( .\right) ,\vartheta
}\left( A\cap B\left( x_{0},r\right) ,B\left( x_{0},2r\right) \right) }{%
\text{cap}_{p\left( .\right) ,\vartheta }\left( B\left( x_{0},r\right)
,B\left( x_{0},2r\right) \right) }\right) ^{\frac{1}{p\left( x_{0}\right) -1}%
}\frac{dr}{r}
\end{equation*}%
for convenience. Also, we denote the Weiner sum as%
\begin{equation*}
W_{p\left( .\right) ,\vartheta }^{sum}\left( A,x_{0}\right)
=\dsum\limits_{i=0}^{\infty }\left( \frac{\text{cap}_{p\left( .\right)
,\vartheta }\left( A\cap B\left( x_{0},2^{-i}\right) ,B\left(
x_{0},2^{1-i}\right) \right) }{\text{cap}_{p\left( .\right) ,\vartheta
}\left( B\left( x_{0},2^{-i}\right) ,B\left( x_{0},2^{1-i}\right) \right) }%
\right) ^{^{\frac{1}{p\left( x_{0}\right) -1}}}.
\end{equation*}%
The Weiner sum is more useful than type integral one in most cases. Now we
give a relationship between these two notions.

\begin{theorem}
\label{denklikthin}Assume that the hypotheses of Theorem \ref{eskap2} and
Theorem \ref{eskap1} are hold. Then there exist constants $C_{1},C_{2}$ such
that%
\begin{equation*}
C_{1}W_{p\left( .\right) ,\vartheta }\left( A,x_{0}\right) \leq W_{p\left(
.\right) ,\vartheta }^{sum}\left( A,x_{0}\right) \leq C_{2}W_{p\left(
.\right) ,\vartheta }\left( A,x_{0}\right)
\end{equation*}%
for every $A\subset 
\mathbb{R}
^{n}$ and $x_{0}\notin A.$ In particular, $W_{p\left( .\right) ,\vartheta
}\left( A,x_{0}\right) $ is finite if and only if $W_{p\left( .\right)
,\vartheta }^{sum}\left( A,x_{0}\right) $ is finite.
\end{theorem}

\begin{proof}
Using the same methods in the Theorem \ref{eskap1} and Theorem \ref{eskap2},
it is easy to see for $r\leq s\leq 2r$ that%
\begin{equation*}
\text{cap}_{p\left( .\right) ,\vartheta }\left( A\cap B\left( x_{0},r\right)
,B\left( x_{0},2r\right) \right) \approx \text{cap}_{p\left( .\right)
,\vartheta }\left( A\cap B\left( x_{0},r\right) ,B\left( x_{0},2s\right)
\right)
\end{equation*}%
and%
\begin{equation*}
\text{cap}_{p\left( .\right) ,\vartheta }\left( B\left( x_{0},r\right)
,B\left( x_{0},2r\right) \right) \approx \text{cap}_{p\left( .\right)
,\vartheta }\left( B\left( x_{0},s\right) ,B\left( x_{0},2s\right) \right)
\end{equation*}%
where the constants in $\approx $ depend on $r$,$p^{-},p^{+}$, constants of
doubling measure and Poincar\'{e} inequality. Thus for $2^{-1-i}\leq r\leq
2^{-i}$ we have%
\begin{eqnarray*}
&&\frac{\text{cap}_{p\left( .\right) ,\vartheta }\left( A\cap B\left(
x_{0},r\right) ,B\left( x_{0},2r\right) \right) }{\text{cap}_{p\left(
.\right) ,\vartheta }\left( B\left( x_{0},r\right) ,B\left( x_{0},2r\right)
\right) } \\
&\leq &C\frac{\text{cap}_{p\left( .\right) ,\vartheta }\left( A\cap B\left(
x_{0},2^{-i}\right) ,B\left( x_{0},2^{1-i}\right) \right) }{\text{cap}%
_{p\left( .\right) ,\vartheta }\left( B\left( x_{0},2^{-i}\right) ,B\left(
x_{0},2^{1-i}\right) \right) } \\
&\leq &C\frac{\text{cap}_{p\left( .\right) ,\vartheta }\left( A\cap B\left(
x_{0},2r\right) ,B\left( x_{0},4r\right) \right) }{\text{cap}_{p\left(
.\right) ,\vartheta }\left( B\left( x_{0},2r\right) ,B\left( x_{0},4r\right)
\right) }.
\end{eqnarray*}%
Hence we obtain that%
\begin{eqnarray*}
&&W_{p\left( .\right) ,\vartheta }\left( A,x_{0}\right) \\
&=&\tsum\limits_{i=0}^{\infty }\tint\limits_{2^{-1-i}}^{2^{-i}}\left( \frac{%
\text{cap}_{p\left( .\right) ,\vartheta }\left( A\cap B\left( x_{0},r\right)
,B\left( x_{0},2r\right) \right) }{\text{cap}_{p\left( .\right) ,\vartheta
}\left( B\left( x_{0},r\right) ,B\left( x_{0},2r\right) \right) }\right) ^{%
\frac{1}{p\left( x_{0}\right) -1}}\frac{dr}{r} \\
&\leq &C\tsum\limits_{i=0}^{\infty }\left( \frac{\text{cap}_{p\left(
.\right) ,\vartheta }\left( A\cap B\left( x_{0},2^{-i}\right) ,B\left(
x_{0},2^{1-i}\right) \right) }{\text{cap}_{p\left( .\right) ,\vartheta
}\left( B\left( x_{0},2^{-i}\right) ,B\left( x_{0},2^{1-i}\right) \right) }%
\right) ^{\frac{1}{p\left( x_{0}\right) -1}} \\
&=&CW_{p\left( .\right) ,\vartheta }^{sum}\left( A,x_{0}\right) .
\end{eqnarray*}%
In a similar way we find%
\begin{eqnarray*}
&&W_{p\left( .\right) ,\vartheta }^{sum}\left( A,x_{0}\right) \\
&\leq &\tsum\limits_{i=0}^{\infty }\tint\limits_{2^{-1-i}}^{2^{-i}}\left( 
\frac{\text{cap}_{p\left( .\right) ,\vartheta }\left( A\cap B\left(
x_{0},2^{-i}\right) ,B\left( x_{0},2^{1-i}\right) \right) }{\text{cap}%
_{p\left( .\right) ,\vartheta }\left( B\left( x_{0},2^{-i}\right) ,B\left(
x_{0},2^{1-i}\right) \right) }\right) ^{\frac{1}{p\left( x_{0}\right) -1}}%
\frac{dr}{r} \\
&\leq &C\tint\limits_{0}^{1}\left( \frac{\text{cap}_{p\left( .\right)
,\vartheta }\left( A\cap B\left( x_{0},2r\right) ,B\left( x_{0},4r\right)
\right) }{\text{cap}_{p\left( .\right) ,\vartheta }\left( B\left(
x_{0},2r\right) ,B\left( x_{0},4r\right) \right) }\right) ^{\frac{1}{p\left(
x_{0}\right) -1}}\frac{dr}{r} \\
&\leq &CW_{p\left( .\right) ,\vartheta }\left( A,x_{0}\right) .
\end{eqnarray*}%
This completes the proof.
\end{proof}

Theorem \ref{denklikthin} give us an equivalent claim for $\left( p\left(
.\right) ,\vartheta \right) $- thinness at $x_{0}.$

\begin{theorem}
Assume that $A\subset 
\mathbb{R}
^{n}$ and $x_{0}\notin A$.

\begin{enumerate}
\item[\textit{(i)}] If $A$ is $\left( p\left( .\right) ,\vartheta \right) $-
thin at $x_{0},$ there exist an open neighborhood $U$ of $A$ such that $U$
is $\left( p\left( .\right) ,\vartheta \right) $- thin at $x_{0}.$

\item[\textit{(ii)}] If $A$ is a Borel set and $\left( p\left( .\right)
,\vartheta \right) $- thick at $x_{0},$ there exist a compact set $K\subset
A\cup \left\{ x_{0}\right\} $ such that $K$ is $\left( p\left( .\right)
,\vartheta \right) $- thick at $x_{0}.$
\end{enumerate}
\end{theorem}

\begin{proof}
Firstly we denote $B_{i}=B\left( x_{0},2^{1-i}\right) .$ Assume that $V_{1}$
and $V_{2}$ are $\left( p\left( .\right) ,\vartheta \right) $- thin at $%
x_{0}.$ By the subadditivity property of relative $\left( p\left( .\right)
,\vartheta \right) $- capacity, it is clear that $V_{1}\cup V_{2}$ is $%
\left( p\left( .\right) ,\vartheta \right) $- thin at $x_{0}.$ Since $%
x_{0}\notin A$ and $x_{0}$ is centers of the balls $B_{i}$ for each $i,$ we
may assume that $A\cap \partial B_{i}=\emptyset .$ Moreover, let $U_{0}=%
\mathbb{R}
^{n}$ and for each $i=1,2,...$ take an open set $U_{i}\subset B_{i}\cap
U_{i-1}$ such that $A_{i}=A\cap B_{i}\subset U_{i}$ and that%
\begin{equation*}
\left( \frac{\text{cap}_{p\left( .\right) ,\vartheta }\left(
U_{i},B_{i-1}\right) }{\text{cap}_{p\left( .\right) ,\vartheta }\left(
B_{i},B_{i-1}\right) }\right) ^{\frac{1}{p\left( x_{0}\right) -1}}\leq
\left( \frac{\text{cap}_{p\left( .\right) ,\vartheta }\left(
A_{i},B_{i-1}\right) }{\text{cap}_{p\left( .\right) ,\vartheta }\left(
B_{i},B_{i-1}\right) }\right) ^{\frac{1}{p\left( x_{0}\right) -1}}+2^{-i-1}.
\end{equation*}%
Let us denote $U=\dbigcup\limits_{i=0}^{\infty }\left( U_{i}-\overline{%
B_{i+1}}\right) .$ Then we obtain that $A\subset U,$ $U$ is open, and%
\begin{eqnarray*}
W_{p\left( .\right) ,\vartheta }^{sum}\left( U,x_{0}\right) &\leq
&\dsum\limits_{i=0}^{\infty }\left( \frac{\text{cap}_{p\left( .\right)
,\vartheta }\left( U_{i},B_{i-1}\right) }{\text{cap}_{p\left( .\right)
,\vartheta }\left( B_{i},B_{i-1}\right) }\right) ^{\frac{1}{p\left(
x_{0}\right) -1}} \\
&\leq &W_{p\left( .\right) ,\vartheta }^{sum}\left( A,x_{0}\right) +1<\infty
.
\end{eqnarray*}%
This completes the proof of \textit{(i)} because of the fact that $U$ is the
desired neighborhood of $A.$

Now we consider the proof of \textit{(ii)}. Again we denote $B_{i}=B\left(
x_{0},2^{1-i}\right) .$ Since the sets $A\cap B_{i}$ are Borel%
\begin{equation*}
\text{cap}_{p\left( .\right) ,\vartheta }\left( A\cap B_{i},B_{i-1}\right)
=\sup_{\substack{ K\subset A\cap B_{i}  \\ compact}}\text{cap}_{p\left(
.\right) ,\vartheta }\left( K,B_{i-1}\right)
\end{equation*}%
for all $i\in 
\mathbb{N}
.$ For each $i$ take a compact $K_{i}\subset A\cap B_{i}$ such that%
\begin{equation*}
\left( \frac{\text{cap}_{p\left( .\right) ,\vartheta }\left(
A_{i},B_{i-1}\right) }{\text{cap}_{p\left( .\right) ,\vartheta }\left(
B_{i},B_{i-1}\right) }\right) ^{\frac{1}{p\left( x_{0}\right) -1}}\leq
\left( \frac{\text{cap}_{p\left( .\right) ,\vartheta }\left(
K_{i},B_{i-1}\right) }{\text{cap}_{p\left( .\right) ,\vartheta }\left(
B_{i},B_{i-1}\right) }\right) ^{\frac{1}{p\left( x_{0}\right) -1}}+2^{-i}.
\end{equation*}%
Hence $K=\dbigcup\limits_{i=0}^{\infty }K_{i}\cup \left\{ x_{0}\right\} $ is
the desired compact set.
\end{proof}

\end{document}